\documentclass[reqno]{amsart}
\usepackage{graphics}
\usepackage{graphicx} 
\usepackage{psfrag}
\psfrag{A}{\fontsize{24}{24}$\psi_0^\ve$}
 \psfrag{B}{\fontsize{24}{24}$\ov{\psi_0^\ve}$}

\psfrag{1}{\fontsize{16}{16}$k_n$}
\psfrag{2}{\fontsize{16}{16}$k_{n-1}$}
\psfrag{3}{\fontsize{16}{16}$k_{n-2}$}
\psfrag{n}{\fontsize{16}{16}$k_1$}
\psfrag{o}{\fontsize{24}{24}$\mathcal{O}_\varepsilon$}
\psfrag{j}{\fontsize{20}{20}$t$}
\psfrag{q}{\fontsize{16}{16}$v + \xi / 2$}
\psfrag{w}{\fontsize{16}{16}$v - \xi / 2$}
\psfrag{r}{\fontsize{16}{16}$v - \xi / 2 + \sum_{l=1}^n k_l$}
\psfrag{0}{\fontsize{16}{16}$0$}
\psfrag{e}{\fontsize{16}{16}$\varepsilon$}
\psfrag{-e}{\fontsize{16}{16}$-\varepsilon$}
\psfrag{c}{\fontsize{16}{16}$v + \xi / 2$}
\psfrag{d}{\fontsize{16}{16}$v - \xi / 2$}
\psfrag{t1}{\fontsize{16}{16}$t_1$}
\psfrag{t2}{\fontsize{16}{16}$t_2$}
\psfrag{O}{\fontsize{16}{16}$O$}
\psfrag{A}{\fontsize{14}{14}$\om > \f 4 {\gamma^2} \f {\mu + 1} \mu$}
\psfrag{B}{\fontsize{14}{14}$\om = \f 4 {\gamma^2} \f {\mu + 1} \mu$}
\psfrag{C}{\fontsize{14}{14}$\om < \f 4 {\gamma^2} \f {\mu + 1} \mu$}
\psfrag{D}{\fontsize{14}{14}$\om_1$}
\psfrag{E}{\fontsize{14}{14}$\om_2$}

\newtheorem{theorem}{Theorem}[section]
\newtheorem{lemma}[theorem]{Lemma}

\theoremstyle{definition}

\newtheorem{prop}[theorem]{Proposition}

\theoremstyle{remark}

\newtheorem{remark}[theorem]{Remark}
\newcommand{\bes}{{\begin{split}}}
\newcommand{\ees}{{\end{split}}}
\newcommand{\bees}{{\begin{equation}\begin{split}}}
\newcommand{\es}{{\end{split}\end{equation}}}
\newcommand{\erre}{{\mathbb R}}
\newcommand{\sech}{\textrm{sech}}

\newcommand{\bea}{\begin{eqnarray}}

\newcommand{\eea}{\end{eqnarray}}

\newcommand{\be}{\begin{equation}}

\newcommand{\ee}{\end{equation}}

\newcommand{\n}{\noindent}

\newcommand{\f}{\frac}

\newcommand{\ve}{\varepsilon}

\newcommand{\om}{\omega}

\newcommand{\ov}{\overline}

\numberwithin{equation}{section}

\begin{document}
\large

\title{Exactly solvable models and bifurcations: the case of the cubic $NLS$
  with a $\delta$ or a $\delta'$ interaction.}


\author{Riccardo Adami}

\address{Riccardo Adami: Dipartimento di Scienze Matematiche
  ``G.L. Lagrange'', Politecnico di Torino, Italy}

\address
{Diego Noja: Dipartimento di Matematica e Applicazioni, 
Universit\`a di Milano Bicocca, Italy}




\author{Diego Noja}



\begin{abstract} 
We explicitly give all stationary solutions to the focusing cubic NLS
on the line, in the presence of a defect of the type
Dirac's delta or delta prime. The models proves interesting for two
features: first, they are exactly solvable and all quantities can be
expressed in terms of elementary functions. Second, the associated dynamics is
far from being trivial.  In particular, the $NLS$ with a delta prime
potential shows two
symmetry breaking bifurcations: the first concerns the ground states
and was already known. The second emerges on the first excited states,
and up to now had not been revealed. We highlight such bifurcations by
computing the nonlinear and the no-defect limits of the stationary solutions.
\end{abstract}

\maketitle

\section{Introduction}
In recent years, the spectacular development of experimental techniques in
condensed matter and ultracold gases, has greatly increased the interest in
the mathematical modeling of the dynamics of Bose-Einstein condensates, i.e. in the
Gross-Pitaevskii
equation
\begin{equation} \label{gp}
i \partial_t \psi (t,x) \ = \ - \Delta  \psi (t,x) \pm \lambda | \psi (t,x)
|^2 \psi (t,x) + V (x) \psi (t,x), \qquad t \in \erre, \, x \in
\erre^d, \lambda > 0 
\end{equation}
where the unknown $\psi (t, \cdot)$ is the wavefunction of the condensate,
the potential $V$ models the action of an external field, and the cubic term
summarizes the effects of the two-body interactions among the
particles in the condensate. It is well-known that the coupling
constant $\pm \lambda$ of the nonlinearity is proportional to the
scattering length of such a 
two-body interaction.  

The function $V$ can be characterized by the same lengthscale of the
condensate, for instance when it models the trap confining the system,
or by a much shorter lengthscale, for instance when it describes the
effect of an inhomogeneity or of an impurity. 

In this paper we focus on the latter case, restrict our study to the so-called 
cigar-shaped condensates, i.e. effectively one-dimensional systems,
and specialize 
to the choices 
\be \label{choices}
1. ~ V (x) = - \alpha \delta_0, \ \alpha > 0. \qquad 2. ~ V (x) = - \gamma
\delta^\prime_0, \ \gamma > 0.
\ee
While choice 1. can be rigorously described as the action of a
delta-shaped potential to be understood in the sense of distributions,
the formalization of choice 2. requires to make resort to  
the self-adjoint extension theory due to Von Neumann and Krein
(\cite{albeverio, dombroski}). Using such theory, one ends up by
defining the singluar interactions in \eqref{choices} by means of
suitable boundary conditions (see \cite{albeverio,kurasov}).

In the present paper we consider the equations
\begin{equation} \label{problem}
i \partial_t \psi (t) \ = \ H_j \psi (t) - \alpha | \psi (t)
|^2 \psi (t), \qquad t \in \erre, \, j=1,2 \, ,
\end{equation}
where the hamiltonian operator $H_1$ and $H_2$ are defined as follows:
\begin{itemize}
\item $H_1$ is the quantum Hamiltonian operator associated to an
  attractive  Dirac's delta potential. Its precise definition is the following:
given $\alpha > 0$,
\be \begin{split} \label{dom1}
D (H_1) \ & : = \ \{ \psi \in H^2 (\erre \backslash \{0 \})\cup H^1 (\erre), 
\ \psi' (0+) - \psi' (0-) =  \alpha \psi (0+)  =  \alpha \psi (0-) \} \\
& (H_1 \psi) (x) \  = \ - \psi^{\prime \prime} (x). 
\end{split}
\ee

\item $H_2$ is the quantum Hamiltonian operator associated to an
  attractive  delta prime potential. Given $\gamma > 0$, 
\be \begin{split} \label{dom2}
D (H_2) \ & : = \ \{ \psi \in H^2 (\erre \backslash \{0 \}), \quad \psi'(0+) = \psi' (0-), \quad
\ \psi (0+) - \psi (0-) =  - \gamma \psi' (0+) \} \\
& (H_2 \psi) (x) \  = \ - \psi^{\prime \prime} (x) .
\end{split}
\ee
\end{itemize}
Thinking of the applications to models for BEC, the natural
mathematical environment is the {\em energy space}. It is widely
known (see e.g. \cite{an-esistenziale})
that the energy space for the Dirac's delta case is 
\be \label{qdelta}
 Q_{\delta} : = H^1 (\erre) 
\ee
and, given $\psi \in Q_{\delta}$, the  related energy functional reads
\be \label{e-delta}
E_{\alpha \delta} ( \psi ) \ : = \ \f 1 2 \int_\erre | \psi' (x) |^2
\, dx - \f \alpha
2 | \psi (0) |^2 - \f \lambda 4 \int_\erre  | \psi (x) |^4
\, dx.
\ee
On the other hand, for the delta prime interaction the energy space
reads
\be \label{qdelta'}
Q_{\delta'} : = H^1 (\erre^+) \oplus H^1 (\erre^-).
\ee
and, given $\psi \in Q_{\delta'}$, the  related energy functional reads
\be \begin{split} \label{e-delta'}
E_{\gamma \delta'} ( \psi ) \ : = & \ \f 1 2 \lim_{\ve \to 0} \left[
  \int_{-\infty}^{-\ve} | \psi' (x) 
  | ^ 2 \, dx + \int_{\ve}^{+\infty} | \psi' (x)
  | ^ 2 \, dx \right] - \f 1 {2 \gamma} | \psi (0+) - \psi (0-) |^2 \\
& \ - \f \lambda 4 \int_\erre  | \psi (x) |^4
\, dx.
\end{split} \ee
In \cite{an-esistenziale} it was proven that the problem given
by \eqref{problem} is globally well-posed for any initial data in the related energy
space. Moreover, $L^2$-norm and energy are
conserved quantities.


As already stressed, the main focus of the present paper is given by
the stationary states of \eqref{problem}.
We call {\em stationary state} any square-integrable solution to the equation
\eqref{problem}
of the type
\be \label{stat}
\psi_\om (t,x) \ = \ e^{i \om t} \phi_\om (x), \qquad \om > 0.
\ee
As a consequence, the 
function $\phi_\om$ must solve the {\em
  stationary Schr\"odinger equation}
\be \label{stat-eq0}
H_j \phi_\om - \lambda | \phi_\om |^2  \phi_\om + \om \phi_\om \ = \ 0, \qquad  j=1,2
\ee
that can be rephrased as the problem of finding a function $\phi_\om$
in the appropriate energy space that satisfies the equation
\be  \label{stat-eq}
- \phi_\om^{\prime \prime}
 - \lambda | \phi_\om |^2  \phi_\om + \om \phi_\om \ = \ 0, \qquad x \neq 0 
\ee
and fulfils the boundary conditions already given in \eqref{dom1} and \eqref{dom2}:
\begin{itemize}
\item
in the case of the Dirac's delta potential
\be \label{bc-delta} 
 \phi_\om' (0+) - \phi_\om' (0-) \ = \ - \alpha \phi_\om (0+)  \ = \ -
 \alpha 
\phi_\om (0-); 
\ee
\item
in the case of the delta prime interaction 
\be \label{bc-delta'} 
 \phi_\om (0+) - \phi_\om (0-) \ = \ - \gamma \phi_\om' (0+)  \ = \ -
 \gamma
\phi_\om' (0-).
\ee
\end{itemize}
It is well-known that any $L^2$ solution to \eqref{stat-eq}
must be of the following type:
\be \label{prototype}
\phi_\om (x) \ : = \ \left\{ \begin{array}{c}
\sqrt{ \f {2\omega} \lambda} \sech (\sqrt \om ( x - x_1)), \qquad x < 0
\\ \
\\
\sqrt{ \f {2\omega} \lambda} \sech (\sqrt \om ( x -  x_2)), \qquad x > 0,
\end{array}
\right.
\ee
so the problem of finding stationary states reduces to the issue of
determining the parameters $x_1$ and $x_2$ in order to fulfil the
matching condtions \eqref{bc-delta} or \eqref{bc-delta'}.

In this paper we find all stationary solutions to equation
\eqref{problem} for both linear Hamiltonians operators $H_1$ and $H_2$
and perform two kinds of limit on them: the so-called {\em linear
  limit}, i.e. $\lambda \to 0+$ and the {\em no-defect limit}, which
amounts to $\alpha \to 0+$ for the Dirac's delta potential; for  
the delta prime interaction it is not so immediate to give a notion of
{\em no-defect limit}, so we study both limits $\gamma \to 0+$, that
seems to be justified as a no-defect limit by \eqref{dom2}, and $\gamma
\to + \infty$ that, according to \eqref{e-delta'} seems to be 
justified as a no-defect limit too.

It is well-known that the standard nonlinear Schr\"odinger equation in
dimension one 
is a completely integrable system (\cite{ZS}); adding a perturbation complete integrability is broken but some characters of exact solvability are retained. For example perturbating with a Dirac's
delta
 interaction, 
the stationary states of the system can still be exactly computed
(\cite{[CM],[GHW], fukuizumi, fukujean, lacozza}). 
In analogy with these cases, the problem of the
stationary states for equation \eqref{problem} with a delta prime
interaction proves 
exactly solvable too. In addition, as shown in \cite{an2}, the structure
of the family of ground 
states can exhibit non-trivial bifurcations. 

The paper is organized as follows: in Section 2 we treat the case of
the Dirac's delta potential, while in Section 3 we study the effect of
a delta prime interaction. Comments on the arising bifurcations are
given after any theorem.

\subsection{Notation}
Along with the symbols already introduced for energy spaces and energy
functionals, we make use of the following notation:
\begin{itemize}
\item
The $L^p (\erre)$-norm of the function $\psi$ is denoted by $\| \psi
\|_p$. The energy norm is referred to as $\| \psi \|_{Q_{\delta}}$
and $\| \psi \|_{Q_{\delta'}}$, for the case with a Dirac's delta and
a delta prime potential, respectively.

\item
We make use of a family of functionals called {\em
  action}, defined as follows: choose $\om > 0$, then
\be \label{action}
S_\om (\psi) \ = E (\psi) + \f \om 2 \| \psi \|^2_2,
\ee
where $E$ can be either $E_{\alpha \delta}$ or $E_{\gamma \delta'}$
according to the case one is examining.
\item For any $A, B > 0, C \in \erre$ we define the function
\be \label{fi}
\phi (A,B,C;x) : = A \, \sech (Bx - C)
\ee
\end{itemize}

\subsection{A preliminary lemma}
The following lemma provides an estimate that will be repeatedly used
along the paper.
\begin{lemma} \label{prelemma}  
Defined the function $\phi(A,B,C)$ as in \eqref{fi}, the following
estimate holds:
\be \begin{split} & \
\| \phi(A,B,C) - \phi(A',B',C') \|_{H^1 (\erre^+) } \\ 
\leq  & \ 
c | A - A'| \| \sech (B \cdot) \|_{H^1 (\erre) } + c | B - B' | \|
e^{- \min (B,B') | \cdot |}  \|_{H^1 (\erre) }
 + c | C - C' |  \|
\sech (B \cdot) \|_{H^2 (\erre) } , 
\end{split}
\ee
where $c$ is a positive constant.
\end{lemma}

\begin{proof}
By triangular inequality
\be
\begin{split}
 & \| \phi (A,B,C) - \phi (A',B',C') \|_{H^1 (\erre^+) } \\
\leq & \ | A - A' | \| \phi (1,B,0) \|_{H^1 (\erre) } + 
| A' | \| \phi (1,B,0) - \phi (1,B',0) \|_{H^1 (\erre) } \\ &  + 
| A' | \| \phi( 1,B',C) - \phi (1,B',C') \|_{H^1 (\erre^+) }
\\
= : & (I) + (II) + (III)
\end{split}
\ee
Term $(I)$ is already in its final form, let us work on 
$(II)$ and $(III)$.

\noindent
Concerning $(II)$, notice that 
$$ | \sech (Bx) - \sech (B'x) | \leq 4 | e^{-Bx} - e^{- B'x} | \leq 4
| B - B' | e^{- \min (B,B') |x|},$$
thus
$$ \| \sech (B \cdot) - \sech (B' \cdot) \|_2^2 \leq 16 | B - B' |^2
\| e^{- \min(B,B') |\cdot| } \|_2^2 . $$
Analogously,
$$ \left| \f d {dx} \left( \sech (Bx) - \sech (B'x) \right) \right|
\leq 6 | e^{-Bx} - e^{- B'x} | 
| B - B' | e^{- \min (B,B') |x|},$$
then
$$ \| \sech (B \cdot) - \sech (B' \cdot) \|_{H^1(\erre)} \leq c | B - B' |
\| e^{- \min (B,B') |\cdot| } \|_{H^1(\erre)} 
$$
Finally, concerning term $(III)$, we observe that
$$ \widehat {\phi(1,B',C)} (k) -  \widehat {\phi(1,B',C)} (k) = \f {e^{-ikC} -
  e^{-ikC'}} {B'} \widehat \sech \left( \f k {B'} \right), $$
where $\widehat f$ denotes the Fourier transform of the function $f$.
Then, a straightforward estimate gives
\be \nonumber
 \| \phi(1,B',C)  -   \phi(1,B',C) \|_{H^1 (\erre)}^2 \leq \f
    {|C - C' |^2}{|B'|^2} \int_\erre (1 + |k|^2) |k|^2 \left|
    \widehat \sech \left( \f k {B'} \right) \right|^2 \, dk.
\ee
The proof is complete.
 \end{proof}

\section{Dirac's delta defect}

Here we investigate the stationary states of the system
\begin{equation} \begin{split}
i \partial_t \psi (t,x) & \ = \ H_1 \psi (t,x) - \lambda |
\psi (t,x) |^2 \psi (t,x) 
\end{split}
\end{equation}
As shown by several authors (\cite{witthaut,fukuizumi,anv13}),
there exists a unique branch of stationary states that reads
\be \label{statdelta}
\psi_{\alpha \delta, \omega} \ = \ \sqrt{\f {2 \omega} \lambda}
\sech (\sqrt \omega (|x| - \bar x_{\alpha, \omega}))
\ee
where
$$
\bar x_{\alpha, \omega} = \f 1 {2 \sqrt \omega} \log \f {2
  \sqrt \omega - \alpha} 
{2 \sqrt \omega + \alpha}.
$$
Notice that $\bar x_{\alpha, \rho, \lambda} < 0$.
\vskip5pt
\noindent
The typical situation is represented in Figure 1.
\vskip5pt

\begin{figure}
\begin{center}
{\includegraphics[width=.40\columnwidth]{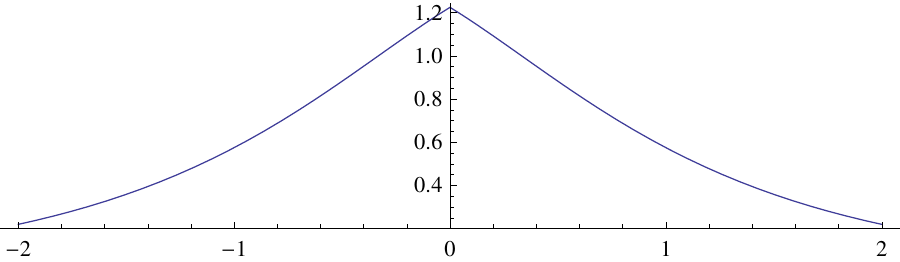}} \quad
{\includegraphics[width=.40\columnwidth]{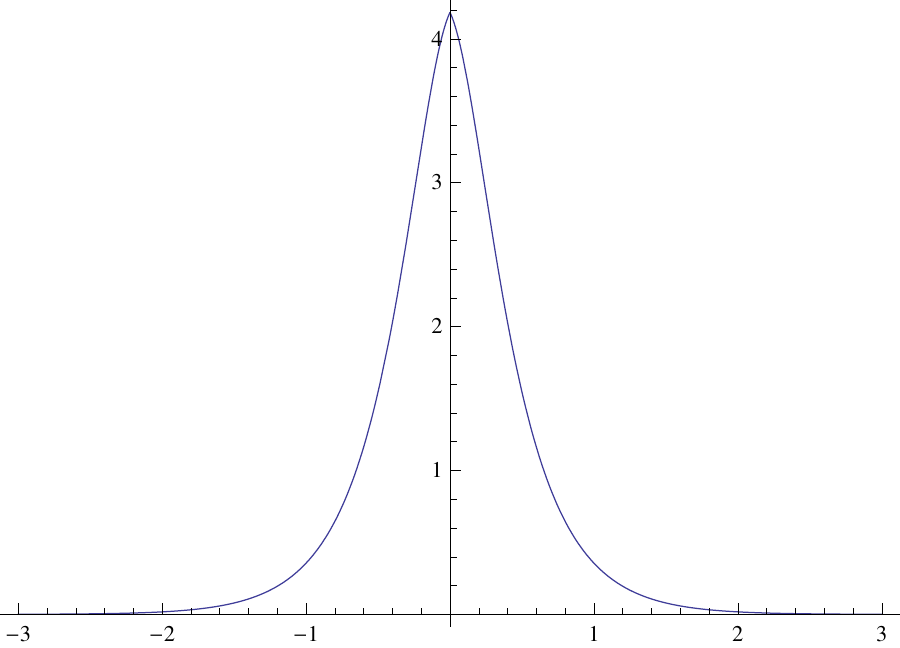}} \\
\caption{ $\psi_{\delta \om}$ for $\alpha=1$, $\lambda=1$ and $\omega=1$ $\omega=9$.}
\end{center}
\end{figure}

By direct computation, one has the following results for $L^2$-norm,
energy and action:
\begin{prop}
The $L^2$-norm, the energy, and the action of the stationary states
given in \eqref{statdelta} read 
\begin{equation} \label{osservabilidelta}
\begin{split}
\| \psi_{\alpha \delta, \omega} \|_2^2 = \f {4 \sqrt \om} \lambda - \f
   {2 \alpha} \lambda & \qquad   E  (\psi_{\alpha \delta, \omega} )
   = - \f 2 3 \f {\om^{\f 3 2}} 
{\lambda} +  \f {\alpha^3} {12 \lambda} \\
S_\omega (\psi_{\alpha \delta, \omega} ) =  \f 4 3 \f {\om^{\f 3 2}}
{\lambda} - \f {\alpha \om} {\lambda} + \f {\alpha^3} {12 \lambda} & 
\end{split}
\end{equation}
\end{prop}
It was proven in \cite{fukuizumi,anv13} that such a branch of stationary states
is indeed made of orbitally stable ground states. Since from
\eqref{osservabilidelta}  the squared $L^2$-norm of the elements of
the branch is a monotonically increasing function of the frequency
$\omega$, it is possible to reparamatrize the family of ground states
in terms of the squared $L^2$-norm, that in the following will be
denoted by $\rho$, just by replacing
$$
\sqrt \om \ = \ \f {\lambda \rho} 4 + \f \alpha 2,
$$
so one obtains
\be 
\label{gsdeltarho} \begin{split}
\varphi_{\alpha \delta, \rho}  (x) 
\ & = \ \psi_{\alpha \delta, \omega
  (\rho)} (x) \ = \ \left( \f {\sqrt \lambda \rho} {2 \sqrt 2} + \f \alpha
    {\sqrt {2\lambda} } 
\right) \sech \left(    \left(\f {\lambda \rho} 4 + \f \alpha 2
\right)
(x - \bar x_{\alpha, \rho, \lambda})  \right) \\
\bar x_{\alpha, \rho, \lambda}  \ &  = \ x_{\alpha, \om (\lambda,
  \rho)} \ = \ \f 2 {\lambda \rho + 2
  \alpha} \log \f 
{\lambda \rho}{\lambda \rho + 4 \alpha}
\end{split}
\ee

\subsection{Linear limit}
Here we study the linear limit  {\em at
  constant mass} of the of the ground state \eqref{gsdeltarho}.
\begin{theorem} \label{22}
The following limit holds
\be \label{linlimdelta}
\varphi_{\alpha \delta, \rho} \longrightarrow \sqrt {\frac {\rho
    \alpha} 2} e^{- \f \alpha 2 |\cdot|} \quad {\rm{in}} \ H^1(\erre),
\quad  {\rm{as}} \ \lambda \to 0+.
\ee
\end{theorem}

\begin{proof}
By triangular inequality,
\be \begin{split} \label{iltriangolono}
& \left\| \varphi_{\alpha \delta, \rho} - \sqrt {\frac {\rho
    \alpha} 2} e^{- \f \alpha 2 |\cdot|} \right\|_{ H^1(\erre)} \\
\leq & 
\left\| \varphi_{\alpha \delta, \rho} - \phi \left(\f \alpha {\sqrt 2
    \lambda}, \f \alpha 2, \bar x_{\alpha, \rho, \lambda} \right)
\right\|_{ H^1(\erre)} + \left\|  \phi \left(\f \alpha {\sqrt 2
    \lambda}, \f \alpha 2, \bar x_{\alpha, \rho, \lambda} \right)
- \varphi_1 \right\|_{ H^1(\erre)} \\
&  + \left\| \varphi_1 - \sqrt {\frac {\rho
    \alpha} 2} e^{- \f \alpha 2 |\cdot|} \right\|_{ H^1(\erre)} \\
= & (I) + (II) + (III)
\end{split} \ee
where 
$
 \varphi_1 (x)   : =   {\sqrt {\f 2 \lambda}} \alpha e^{-
   \f \alpha 2 (|x| - 
   \bar x_{\alpha, \rho, \lambda} )}. 
$
Noting that $\varphi_{\alpha \delta, \rho} = \phi \left(
\frac {\sqrt \lambda \rho} {2 \sqrt 2} + \f \alpha {\sqrt{2 \lambda}}, 
\f {\lambda \rho} 4 + \f \alpha 2, \bar x_{\alpha, \rho, \lambda} \right),$
term (I) can be estimated by Lemma \ref{prelemma} as
$$
(I) \ \leq \ c \rho \sqrt {\lambda} \left\| \sech \left(
\f \alpha 2 \cdot \right)
\right\|_{H^1 (\erre)} + c \rho \lambda    \left\| e^{- \f \alpha 2 |
  \cdot | }
\right\|_{H^1 (\erre)}.
$$
Concerning term $(II)$, by straightforward computations one has
\be \begin{split} \label{II.a}
\left\|    \phi \left(\f \alpha {\sqrt 2
    \lambda}, \f \alpha 2, \bar x_{\alpha, \rho, \lambda} \right)
- \varphi_1 
\right\|_2^2 \ = & \
 \f {4 {\alpha^2}} \lambda \int_0^{+\infty} \f {e^{-  \alpha  (
  x - \bar x_{\alpha, \rho, \lambda})}}{ (e^{ \alpha  (
  x - \bar x_{\alpha, \rho, \lambda})} + 1            )^2 } dx \\
\leq & \   \f {4 \alpha^2} \lambda e^{3  \alpha  (
  x - \bar x_{\alpha, \rho, \lambda})} \int_0^{+\infty} e^{-3 \alpha
  x} \, dx \
\leq \  c \rho^3 \alpha^2 \lambda^2,
\end{split} 
\ee
where, in the last line, we used the definition of $\bar x_{\alpha, \rho,
  \lambda}$. 
Furthermore,
\be \begin{split} \label{II.b}
& \left\| \f d {dx} \left(   \phi \left(\f \alpha {\sqrt 2
    \lambda}, \f \alpha 2, \bar x_{\alpha, \rho, \lambda} \right)
- \varphi_1 \right)
\right\|_2^2 
\\
= & \  \f {4 \alpha^4} \lambda \int_0^{+\infty} \left|\f
{e^{-  \alpha  (
  x - \bar x_{\alpha, \rho, \lambda})}} { e^{ \alpha  (
  x - \bar x_{\alpha, \rho, \lambda})} + 1             } +
\f {e^{ \alpha  (
  x - \bar x_{\alpha, \rho, \lambda})}} { (e^{ \alpha  (
  x - \bar x_{\alpha, \rho, \lambda})} + 1            )^2 }
\right|^2 \, dx \\
\leq & \ c \alpha^4  \int_0^{+\infty} e^{-  \alpha  (
  x - \bar x_{\alpha, \rho, \lambda})} \, dx \ \leq \ c \alpha^4
\rho^3 \lambda^2.
\end{split} \ee
Concerning $(III)$, it is sufficient to notice that
$$ \varphi_1 (x) = \sqrt{\f 2 \lambda} \alpha \left( \f {\lambda \rho} {{\lambda \rho + 4
    \alpha}}  \right)^{\f \alpha  {\lambda \rho + 2
    \alpha}} e^{- \f \alpha 2 |x|},
$$
and since 
$$ \sqrt{\f 2 \lambda} \alpha \left( \f {\lambda \rho} {{\lambda \rho + 4
    \alpha}}  \right)^{\f \alpha  {\lambda \rho + 2
    \alpha}} \ \longrightarrow \sqrt{\f {\rho \alpha} 2} \quad
{\rm{as}} \ \lambda \to 0+
$$
we have that term $(III)$ vanishes too and the proof is complete.
\end{proof}

\subsection{No-defect limit}
\begin{theorem}
The following limit holds:
$$
\varphi_{\alpha \delta, \rho} \longrightarrow \phi \left( \f {\lambda
  \rho} {2 \sqrt 2}, \f {\lambda
  \rho} 4, 0 \right) ~ {\rm{in}} \ H^1 (\erre), \qquad \alpha \to 0+.
$$
\end{theorem}

\begin{proof}
The result immediately follows from Lemma \ref{prelemma}, once observed
that 
$$\varphi_{\alpha \delta, \rho} = \phi \left(
\frac {\sqrt \lambda \rho} {2 \sqrt 2} + \f \alpha {\sqrt {2 \lambda}}, 
\f {\lambda \rho} 4 + \f \alpha 2, \bar x_{\alpha, \rho, \lambda} \right).$$
\end{proof}

\section{Delta prime defect}
Here we investigate the stationary states of the system
\begin{equation} \begin{split}
i \partial_t \psi (t,x) & \ = \ H_2 \psi (t,x) - \lambda |
\psi (t,x) |^2 \psi (t,x) .
\end{split}
\end{equation}
To this purpose,
we reduce equation
\eqref{stat-eq} with boundary conditions \eqref{bc-delta'} to a couple
of systems of two equations in two real 
unknowns. For the convenience of the reader, we preliminarily solve such
systems.

\subsection{Two algebraic systems}
We preliminary solve a couple of systems that will be used in order to
explicitly write down the stationary solutions to \eqref{problem}, \eqref{bc-delta'}.

\begin{prop} \label{3.1} For any $\om > 0$,
the only solution $(t_{1}, t_{2})$ to the system
\be 
\label{tipiu}
T_+ : = \left\{  \begin{array}{ccc}
t_1^2 - t_1^4 & = & t_2^2 - t_2^ 4
\\ & & \\
t_1^{-1} - t_2^{-1} & = & \gamma \sqrt \om
\end{array} \right.
\ee
satisfying the condition $0 \leq t_1, t_2 \leq 1$, is given by
\be \begin{split} \label{solt+}
t_1 \ = & \ \f{1 - \sqrt{1 + \gamma^2 \om} + \sqrt{ \gamma^2 \om - 2 +
    2 \sqrt{1 + \gamma^2 \om}}} 
{2 \gamma \sqrt \om} \\
t_2  \ = & \ \f{- 1 + \sqrt{1 + \gamma^2 \om} + \sqrt{ \gamma^2 \om -
    2 + 2 \sqrt{1 + \gamma^2 \om}}} 
{2 \gamma \sqrt \om}
\end{split} \ee
\end{prop}
\begin{proof} 
Along the proof we consider the case $\gamma = 1$. The generic case
can be recovered by the replacement $\omega \to \gamma^2 \omega$.

\noindent
The first equation in \eqref{tipiu} rewrites as
$$
(t_1^2 - t_2^2) (t_1^2 + t_2^2 - 1) \ = \ 0,
$$
so the system splits into two subsystems $T_{1,+}$ and $T_{2,+}$, defined by
\be T_{1,+} : = \left\{  \begin{array}{ccc}
t_1 & = & t_2
\\ & & \\
t_1^{-1} - t_2^{-1} & = & \sqrt \om
\end{array}
\right. \nonumber
\ee
and
\be T_{2,+} : = \left\{  \begin{array}{ccc}
t_1^2 + t_2^2 & = & 1
\\ & & \\
t_1^{-1} - t_2^{-1} & = & \sqrt \om
\end{array}
\right. \nonumber
\ee
The system $T_{1,+}$ can be solved only for $\om = 0$, so we are not
interested  in such a solution

\noindent
Let us consider the system $T_{2,+}$. Multiplying the second equation by $t_1t_2$ 
one obtains
\be \label{summa}
t_2 - t_1 \ = \ \sqrt \om t_1 t_2,
\ee
thus squaring both sides, using the first equation of $T_2$ and
imposing positivity,   we obtain
$$
t_1 t_2 \ = \ \f{\sqrt{1 +  \om} - 1} {\om}
$$
and so, by \eqref{summa}
\be \label{strauss}
t_2 - t_1 \ = \ \f{\sqrt{1 +  \om} - 1}{\sqrt \om}.
\ee
By \eqref{strauss} and the first equation of $T_{2,+}$, we finally get
\eqref{solt+}.
%
\end{proof}

\begin{prop}
The solutions $(t_1, t_2)$ to the system
\be 
\label{timeno}
T_- : = \left\{  \begin{array}{ccc}
t_1^2 - t_1^4 & = & t_2^2 - t_2^ 4
\\ & & \\
t_1^{-1} + t_2^{-1} & = & \gamma \sqrt \om
\end{array} \right. \ee
with the condition $0 \leq t_1, t_2 \leq 1$, can be classified as
follows: 
\begin{enumerate}
\item For $0 < \om \leq \f 4 {\gamma^2}$ there are no solutions.
\item For $\om > \f 4 {\gamma^2}$ there exists the solution
\be \label{t-symm}
t_1 \ = \ t_2 \ = \ \f 2 {\gamma \sqrt \om}.
\ee
\item For $\om > \f 8 {\gamma^2}$ there exist two solutions $(t_{1}, t_{2})$,
  and $(t_{2}, t_{1})$, where
\be \begin{split} \label{solt-asym}
t_{1} \ = & \ \f{1 + \sqrt{1 + \gamma^2 \om} + \sqrt{ \gamma^2 \om - 2
    - 2  \sqrt{1 + \gamma^2 \om}}}
{2 \gamma \sqrt \om} \\
t_{2}  \ = & \ \f{ 1 + \sqrt{1 + \gamma^2 \om} - \sqrt{ \gamma^2 \om -
    2 - 2  \sqrt{1 + \gamma^2 \om}}}
{2 \gamma \sqrt \om}
\end{split}
\ee
\end{enumerate}
\end{prop}
\begin{proof}
As in the previous proof, we put $\gamma = 1$ and to write the final
result we perform the change $\omega \to \gamma^2 \omega$.
 We rewrite the first equation in
\eqref{solt-asym} as
$$
(t_1^2 - t_2^2) (t_1^2 + t_2^2 - 1) \ = \ 0,
$$
so the system \eqref{solt-asym} is equivalent to the union of 
the systems
\be T_{1,-} : = \left\{  \begin{array}{ccc}
t_1 & = & t_2
\\ & & \\
t_1^{-1} + t_2^{-1} & = & \sqrt \om
\end{array}
\right. \nonumber
\ee
and
\be T_{2,-} : = \left\{  \begin{array}{ccc}
t_1^2 + t_2^2 & = & 1
\\ & & \\
t_1^{-1} + t_2^{-1} & = & \sqrt \om
\end{array}
\right. \nonumber
\ee
The system ${T_{1,-}}$ gives the solution \eqref{t-symm}. The condition $\om >
4$ emerges by the prescription $t_1, t_2 < 1$, so point {\em 1.} is
proven. 

Let us consider the system $T_{2,-}$. Multiplying the second equation by $t_1t_2$ 
one obtains
\be \label{summa2}
t_1 + t_2 \ = \  \sqrt \om t_1 t_2,
\ee
thus squaring both sides, using the first equation of $T_{2,-}$ and
imposing positivity,   we have
$$
t_1 t_2 \ = \ \f{1 + \sqrt{1 +  \om}} { \om}
$$
and so, by \eqref{summa2}
\be \label{strauss2}
t_1 + t_2 \ = \ \f{1 + \sqrt{1 +  \om}}{ \sqrt \om}.
\ee
By \eqref{strauss2} and the first equation of ${T_{2,-}}$, we finally get
\eqref{solt-asym}. The positivity of the quantity under the square
root of \eqref{solt-asym} gives the condition $\om > 8$, so the proof
is complete.
\end{proof}

\subsection{Stationary states for the case with a $\delta'$ potential}
 
In the following theorem we explicitly give all solutions to
\eqref{stat-eq}-\eqref{bc-delta'}.

\begin{theorem}\label{statstat}
The stationary states of equation \eqref{problem} can be classified as
follows:
\begin{enumerate}
\item The unperturbed symmetric solitary wave
\be \label{unperturbed} 
\psi_{\om}^{0,0}  (x) \ : = \ 
\sqrt{\f{2 \om} \lambda} \sech (\sqrt\om x).
\ee
Such solutions are present for any $\om > 0$.

\item The asymmetric non-changing sign solutions
\be \label{stazzionario2+} 
\psi_{\om, +}^{x_{1\pm}, x_{2\pm}} (x) \ : = \ \left\{ 
\begin{array}{ccc}
\sqrt{\f{2 \om} \lambda} \sech (\sqrt\om (x-x_{1 \pm})), & & x< 0 \\ & & \\
\sqrt{\f{2 \om} \lambda} \sech (\sqrt\om (x-x_{2 \pm})), & & x> 0,
\end{array} \right.
\ee
\be \begin{split} \label{313}
x_{1+} \ = & \ \f 1 {2 \sqrt \om} \log \f{1+ t_1} {1 - t_1}, \quad
x_{2+} \ =  \ \f 1 {2 \sqrt \om} \log \f{1+ t_2} {1 - t_2} \\ \\
x_{1-} \ = & \ - \f 1 {2 \sqrt \om} \log \f{1+ t_2} {1 - t_2}, \quad 
x_{2-} \ =  \ - \f 1 {2 \sqrt \om} \log \f{1+ t_1} {1 - t_1}
\end{split}, \ee
where the couple $(t_1, t_2)$ solves the system \eqref{tipiu}.

\noindent
Such solutions are present for any $\om > 0$.

\item  The changing sign solutions, which, in turn, can be
classified as
\begin{enumerate}
\item the antisymmetric solutions
\be \label{stazzionario} 
\psi_{\om, -}^{\bar x, - \bar x} \ : = \ \left\{ 
\begin{array}{ccc}
- \sqrt{\f {2 \om} \lambda} \, \sech (\sqrt\om (x- \bar x)), & & x< 0 \\ & & \\
 \sqrt{\f {2 \om} \lambda} \, \sech (\sqrt\om (x + \bar x)), & & x > 0,
\end{array} \right.
\ee
where 
\be \label{xbar}
\bar x = \f 1 {2 \sqrt \om} \log \f{\gamma \sqrt \om + 2} {\gamma \sqrt \om - 2}.
\ee
Such solutions are present for $\om > \f 4{\gamma^2}$;

\item the asymmetric solutions
\be \label{stazzionario2-} 
\psi_{\om, -}^{\bar x_{1 \pm}, \bar x_{2 \pm}} \ : = \  \left\{ 
\begin{array}{ccc}
\sqrt{\f {2 \om} \lambda} \, \sech (\sqrt\om (x-\bar x_{1 \pm})), & & x< 0 \\  & & \\
- \sqrt{\f {2 \om} \lambda} \, \sech (\sqrt\om (x-\bar x_{2 \pm})), & & x> 0,
\end{array} \right.
\ee

\be \label{317} \begin{split}
\bar x_{1+} \ = & \ \f 1 {2 \sqrt \om} \log \f{1+ t_1} {1 - t_1}, \quad
\bar x_{2+} \ =  \ - \f 1 {2 \sqrt \om} \log \f{1+ t_2} {1 - t_2} \\
\bar x_{1-} \ = & \  \f 1 {2 \sqrt \om} \log \f{1+ t_2} {1 - t_2}, \quad
\bar x_{2-} \ =  \ - \f 1 {2 \sqrt \om} \log \f{1+ t_1} {1 - t_1}
\end{split}, 
\ee
where the couple $(t_1, t_2)$ solves the system \eqref{timeno}.

\noindent
Such solutions are present for any $\om > \f 8 {\gamma^2}$.
\end{enumerate}
\end{enumerate}
\end{theorem}
\vskip5pt
\par\noindent
See Figure 2 and Figure 3 for a representation of the typical behaviours.
\vskip5pt
\begin{proof}
The fact that $\psi_\om^{0,0}$ is a stationary state can be
established by direct computation.

For the other stationary states,
one has to fix the parameters $x_1$ and $x_2$ in
\eqref{prototype}
imposing the matching
conditions \eqref{bc-delta'}. It turns out that there are two families
of solutions, according to the choice of the 
sign of $\phi_\omega$ in the positive halfline: the {\em changing
  sign} and the {\em non changing sign} solutions.

\noindent
-- In order to obtain
 the family of non changing sign solutions $\psi_{\om,+}^{x_{1,\pm},
  x_{2\pm}}$, one defines $\tanh (\sqrt \om x_i) = t_i$ and system
 \eqref{bc-delta'} translates into
the
system
\be T_+ : = \left\{  \begin{array}{ccc}
t_1^2 - t_1^4 & = & t_2^2 - t_2^ 4
\\ & & \\
t_1^{-1} - t_2^{-1} & = & \gamma \sqrt \om
\end{array}
\right. ,  \qquad 0 \leq t_1, t_2 \leq 1. \nonumber
\ee
Then, the possible couples $(x_{1}, x_2)$ are two, that we denote by
$(x_{1+}, x_{2+})$ and $(x_{1-}, x_{2-})$ and are given by \eqref{stazzionario2+},
where the couple $(t_1, t_2)$ is given by \eqref{solt+}. 

\noindent
-- In order to find the family of changing sign solutions
$\psi_{\om,-}^{\bar x_{1\pm},
  \bar x_{2\pm}}$, the parameters $t_i = | \tanh (\sqrt \omega x_i) |$ solve
the
system
\be T_- : = \left\{  \begin{array}{ccc}
t_1^2 - t_1^4 & = & t_2^2 - t_2^ 4
\\ & & \\
t_1^{-1} + t_2^{-1} & = & \gamma \sqrt \om
\end{array}
\right. ,  \qquad 0 \leq t_1, t_2 \leq 1. \nonumber
\ee
Then, 
\eqref{t-symm} immediately provides solutions \eqref{stazzionario}. On
the other hand, equation \eqref{solt-asym} provides two further couples
$(x_{1\pm}, x_{2 \pm})$, defined in \eqref{317}.
\end{proof}

\begin{figure}
\begin{center}
{\includegraphics[width=.40\columnwidth]{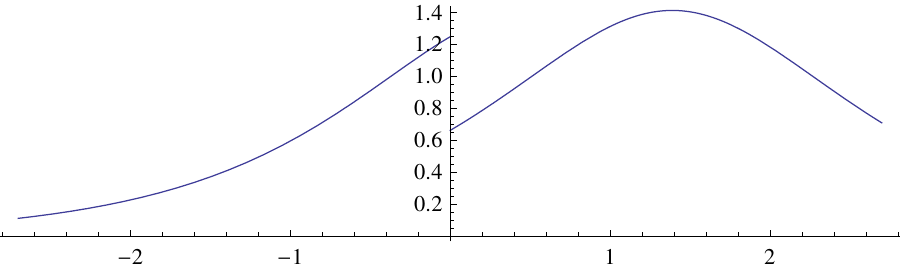}} \quad
{\includegraphics[width=.40\columnwidth]{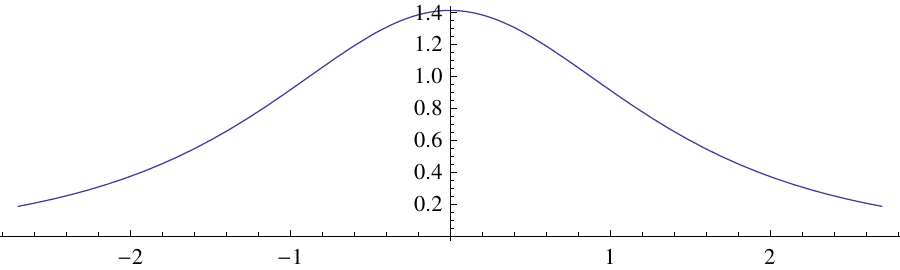}} \\
\caption{ $\psi_{\om, +}^{x_{1+}, x_{2+}}$ for $\gamma=1$,
  $\lambda=1$, $\omega=9$ and pure soliton with $\omega=1$.} 
\end{center}
\end{figure}

\begin{figure}
\begin{center}
{\includegraphics[width=.40\columnwidth]{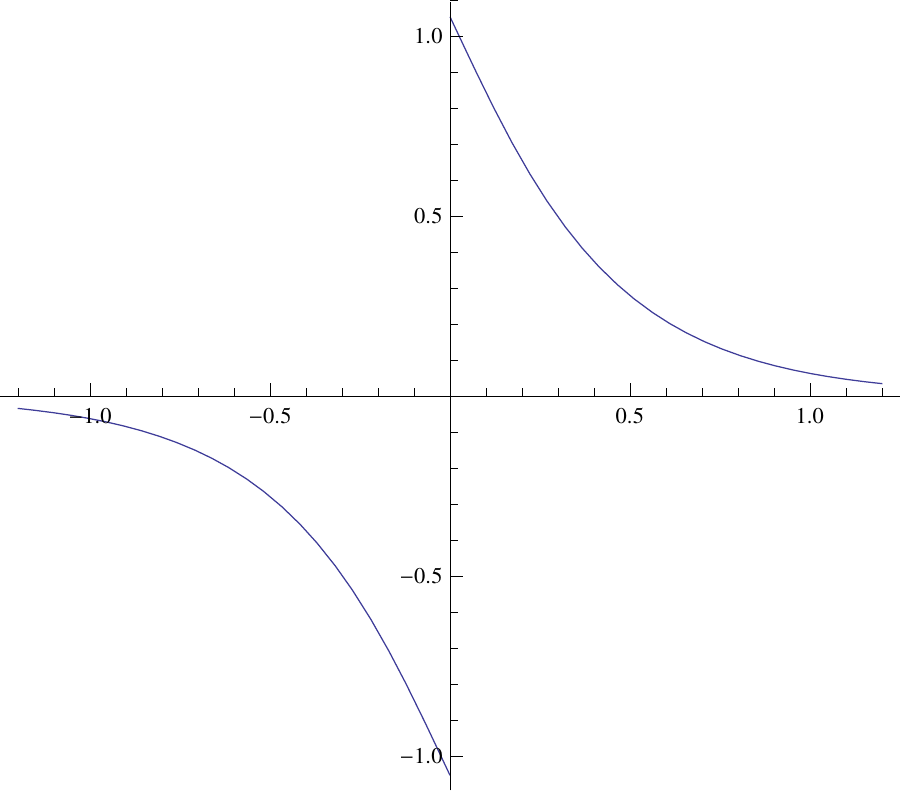}} \quad
{\includegraphics[width=.40\columnwidth]{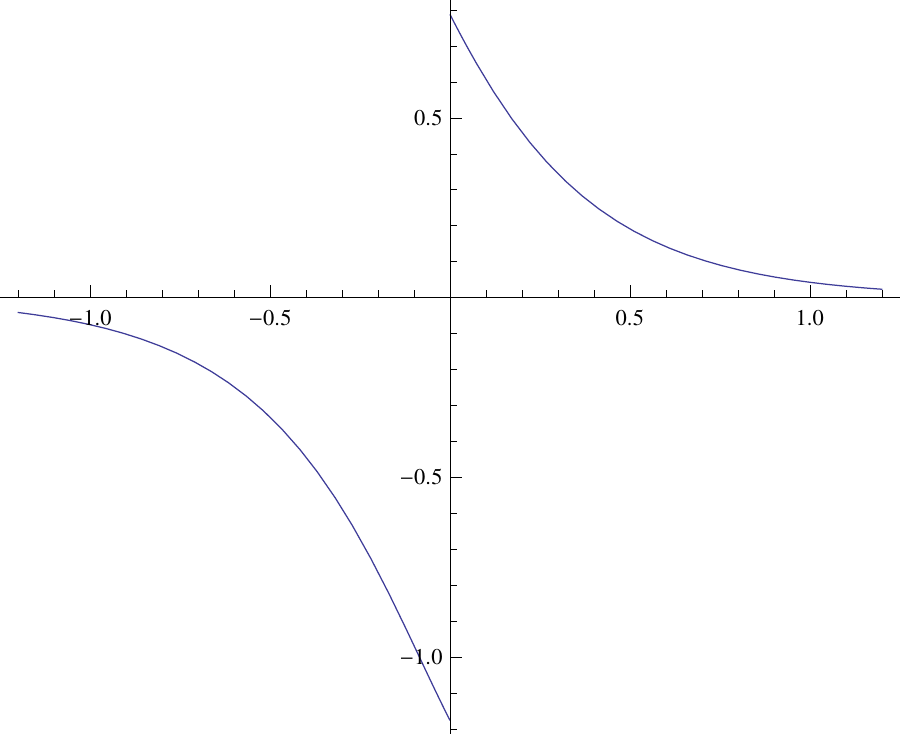}} \\
\caption{ $\psi_{\om, -}^{\bar x, - \bar x}$ and $\psi_{\om, -}^{x_{1+}, x_{2+}}$ for $\gamma=1$, $\lambda=1$, $\omega=9.$}
\end{center}
\end{figure}

By direct comptutations, one can prove the following
\begin{prop}
The $L^2$-norm, 
the energy, and the action of the stationary
states found in Theorem \ref{statstat}, are explicitly given by
\begin{enumerate}
\item For the unperturbed symmetric solitary wave
\be \label{e1} \begin{split}
\| \psi_{\om,+}^{0,0} \|_2^2 \ = \ \f
4 \lambda \sqrt \om, \quad & 
\quad E_{\gamma \delta'} ( \psi_{\om,+}^{0,0} ) \ = \
- \f 2 {3 \lambda} \om^{\f 3 2},  \quad \quad
S_\om ( \psi_{\om,+}^{0,0} ) \ = \
 \f 4 {3 \lambda} \om^{\f 3 2}. \\
\end{split} \ee
\item For the asymmetric non-changing sign stationary states
\be \label{e2} \begin{split}
\| \psi_{\om,+}^{x_{1\pm},x_{2\pm}} \|_2^2 \ = & \ \f 4 \lambda \sqrt \om + \f 2
   {\lambda \gamma} 
\sqrt{1 +
\gamma^2 \om}
- \f 2  {\lambda \gamma} \\
E_{\gamma \delta'}  (\psi_{\om,+}^{x_{1\pm},x_{2\pm}}) \ = & \ - \f 2 {3 \lambda} \om^{\f 3 2} + \f 
{2 - \gamma^2 \om}
{3 \lambda\gamma^3} \sqrt{\gamma^2 \om + 1} - \f 2 {3 \lambda \gamma^3}  \\ 
S_\om (\psi_{\om,+}^{x_{1\pm},x_{2\pm}}) \ = &  \f 4 {3 \lambda} \om^{\f 3 2} +
\f 2 {3 \lambda \gamma^3}
(\gamma^2 \om + 1)^{\f 3 2} - \f \om {\lambda \gamma} - \f 2 {3
  \lambda \gamma^3} \\ 
\end{split} \ee
\item For the antisymmetric solutions
\be \label{e3}
\| \psi_{\om,-}^{\bar x, - \bar x} \|_2^2 \ = \ \f 4 \lambda \sqrt \om
- \f 8 {\gamma \lambda}, \quad 
E_{\gamma \delta'} (\psi_{\om,-}^{\bar x, - \bar x}) \ = \ \f {16} {3 \gamma^3 \lambda}
- \f 2 {3 \lambda} 
 \om^{\f 3 2}, \quad
S_\om (\psi_{\om,-}^{\bar x, - \bar x}) \ = \ \f {16} {3 \gamma^3 \lambda} - 4
\f \om {\gamma \lambda} + \f 4 
{3 \lambda}
\om^{\f 3 2} 
\ee
\item For the asymmetric, changing sign solutions
\be \label{e4} \begin{split}
\| \psi_{\om,-}^{\bar x_{1\pm},\bar x_{2\pm}} \|_2^2 \ = & \ \f 4 \lambda \sqrt \om - 
\f 2 {\lambda \gamma} \sqrt{1 + \gamma^2 \om}
- \f 2 {\lambda \gamma} \\
E_{\gamma \delta'} (\psi_{\om,-}^{\bar x_{1\pm},\bar x_{2\pm}}) \ = & \ - \f 2 {3
  \lambda} \om^{\f 3 2} + \f 
{\gamma^2 \om - 2}
{3 \lambda \gamma^3} (\sqrt{\gamma^2 \om + 1}) - \f 2 {3 \lambda \gamma^3} \\
S_\om (\psi_{\om,-}^{\bar x_{1\pm},\bar x_{2\pm}}) \ = &  \f 4 {3 \lambda} \om^{\f 3 2} +
\f 2 {3 \lambda \gamma^3}
(\gamma^2 \om + 1)^{\f 3 2} - \f \om {\lambda \gamma }- \f 2 {3
  \lambda \gamma^3}  \\ 
\end{split} \ee
\end{enumerate}
\end{prop}
\begin{remark} \label{cor}
In any family of stationary states $\psi_{\om,+}^{0,0}$,
$\psi_{\om,+}^{x_{1 \pm},x_{2 \pm}}$,  $ \psi_{\om,-}^{\bar x, - \bar x}$,  and
$\psi_{\om,-}^{\bar x_{1 \pm}, \bar x_{2 \pm}}$, introduced in Theorem \ref{statstat}, 
the $L^2$-norm is a monotonically increasing
function of the frequency $\om$ that diverges at infinity. In other
words, it is a bijection of the interval $(0, + \infty)$ into itself.
\end{remark}

\subsection{Minimizing the action on the Nehari manifold}
Here we consider the variational problem 
given by minimizing the action $S_\om$ at fixed $
\om > 0$ on the natural constraint called the {\em Nehari manifold}
\be \label{nehari}
E_{\gamma \delta'} (\psi) - \f \lambda 4 \| \psi \|_4^4 \ = \ 0.
\ee

\begin{prop}
1. For any $\om > 0$, 
\be \label{disug1}
S_\om (\psi_{\om,+}^{0,0}) < S_\om
(\psi_{\om,+}^{x_{1\pm},x_{2\pm}}).
\ee

\noindent
2. For any $\om > \f 4 {\gamma^2}$,
\be \label{disug2}
S_\om (\psi_{\om,-}^{\bar x,- \bar x}) < S_\om (\psi_{\om,+}^{0,0}) < S_\om
(\psi_{\om,+}^{x_{1\pm},x_{2\pm}}).
\ee

\noindent
3. For any $\om > \f 8 {\gamma^2}$,
\be \label{disug3}
S_\om (\psi_{\om,-}^{\bar x_{1\pm}, \bar x_{2\pm}}       ) < S_\om
(\psi_{\om,-}^{\bar x,- \bar 
  x}) < 
S_\om (\psi_{\om,+}^{0,0}) < S_\om
(\psi_{\om,+}^{x_{1\pm}, x_{2\pm}}).
\ee
\end{prop}

\begin{proof}
Inequalities \eqref{disug1} and \eqref{disug2} follow immediately from
\eqref{e1}, \eqref{e2}, \eqref{e3}.

\n
The first inequality in
\eqref{disug3} is proven by \eqref{e3} and \eqref{e4}. In particular,
one gets that the inequality is equivalent to 
$$
3 \gamma^2 \om < 6 + \f 2 3 (\gamma^2 \om + 1)^{\f 3 2}
$$
that is immediately verified for any $\om > \f 8 {\gamma^2}$.
\end{proof}

\begin{remark}\label{bifurcations}
Define the function $\Delta_+ (\om) : =
S_\om (\psi_{\om,+}^{x_{1, \pm},x_{2, \pm}}) -  S_\om (\psi_{\om,+}^{0,0})$ and
remark that
$$ \Delta_+ (0) \ = \ \f d {d \om}  \Delta_+ (0) \ = \ 0,$$
while 
$\f {d^2} {d\om^2}  \Delta_+ (\om) \ = \ \f {\gamma^2} 2 
(\gamma^2 \om + 1)^{-1/2}
  \ > \ 0$.

\noindent
This phenomenon suggests that 
the stationary states $\psi_{\om,+}^{x_{1, \pm},x_{2, \pm}}$ bifurcate 
from a stationary state corresponding to the case with no defect,
i.e., states of the type $ \phi (\sqrt {\f {2 \om} \lambda}, \sqrt
\om, x_0)$, for some $x_0$, 
at $\om > 0$.

\noindent

\noindent

\noindent
By an analogous computation one findes that
the stationary states $\psi_{\om,-}^{\bar x_{1\pm}, \bar x_{2 \pm}}$ bifurcate from
$\psi_{\om,-}^{\bar x, -\bar x}$ at $\om > \f 8 {\gamma^2}$.
\end{remark}

\subsection{Minimizing the energy in the manifold of constant mass}
Here we consider the variational problem 
of minimizing the energy $E_{\gamma \delta'}$ on the manifold consisting of the
function with given $L^2$-norm $\sqrt \rho > 0$ and finite energy.
As a consequence of Remark
\ref{cor}, one can parametrize the stationary states by the value
$\rho$ of their squared $L^2$-norm $\rho \in (0, + \infty)$, 
instead of $\omega$. In the following, 
given $\rho > 0$, we
make use of the notation 
$$
\varphi_{\rho, +}^{0,0} \ : = \{ \psi_{\om, +}^{0,0}, \ \om \
\rm{s.t.} \
  \| \psi_{\om, +}^{0,0} \|_2^2 \ = \ \rho \}
$$
and analogously for the other families $\varphi_{\rho,
  +}^{x_{1 \pm},x_{2 \pm}}$, $\varphi_{\rho, -}^{\bar x,-\bar x}$,
and $\varphi_{\rho, -}^{\bar x_{1 \pm}, \bar x_{2 \pm}}$.

\begin{prop}
1. For any $\rho > 0$, 
\be \label{ineq1}
E_{\gamma \delta'} (\varphi_{\rho,-}^{\bar x,- \bar x}) \ < \ E_{\gamma \delta'} (\varphi_{\rho,+}^{0,0})
\ < \ E_{\gamma \delta'}
(\varphi_{\rho,+}^{x_{1 \pm},x_{2 \pm}}).
\ee

\noindent
2. For any $\rho > \f 8 {\gamma \lambda} \left( \sqrt 2 - 1 \right)$,
\be \label{ineq2}
E_{\gamma \delta'}
(\varphi_{\rho,-}^{\bar x_{1 \pm}, \bar x_{2 \pm}}) \ < \
E_{\gamma \delta'} (\varphi_{\rho,-}^{\bar x,- \bar x}) \ < \ E_{\gamma \delta'} (\varphi_{\rho,+}^{0,0})
\ < \ E_{\gamma \delta'}
(\varphi_{\rho,+}^{x_{1 \pm},x_{2 \pm}}).
\ee
\end{prop}

\begin{proof}
In order to make the formulas less cumbersome, we show the
computations for
the particular case $\lambda = 1$, $\gamma = 1$ only. The general
formulas can be then recovered by the scaling laws
\be \label{scaling} 
\om \to \gamma^2 \om, \quad \rho \to \gamma \lambda \rho, \quad
E_{\delta'}  \to
\gamma^3 \lambda E_{\gamma \delta'},
\ee
that can be directly verified.

Then, we make explicit the dependence of $\om$, and then of the
energy, on $\rho$. To this aim, it proves simpler and more practical
to write down the function $\sqrt \om (\rho)$ 
An elementary but lengthy computation shows that
\begin{itemize}
\item For the family $\varphi_{\rho,+}^{0,0}$
\be \label{omro1}
\om \ = \ \f {\rho^2}{16},  \qquad E_{\delta'}(\varphi_{\rho,+}^{0,0}) \ = \ -
\f{\rho^3}{96}. 
\ee
\item For the family $\varphi_{\rho,+}^{x_{1 \pm},x_{2 \pm}}$
\be \label{omro2} \begin{split}
\sqrt \om \ = & \ \f{2 \rho + 4 - \sqrt{\rho^2 + 4 \rho + 16}}{6} \\
\qquad E_{\delta'}(\varphi_{\rho,+}^{x_{1 \pm},x_{2 \pm}}) \ = & \ - \f 5 {216} \rho^3 + \f 1 {54}
( \rho^2 + 4 \rho + 16)^{\f 3 2} - \f 5 {36} \rho^2 - \f 4 9  \rho - \f
{32} {27}.
\end{split} \ee
\item For the family $\varphi_{\rho,-}^{\bar x, - \bar x}$
\be \label{omro3}
\sqrt \om \ = \ \f \rho 4 + 2,  \qquad E_{\delta'}(\varphi_{\rho,-}^{-\bar x,
  \bar x}) \ = \ - \f {\rho^3}{96} - \f {\rho^2}{4} - 2 \rho.
\ee
\item For the family $\varphi_{\rho,-}^{\bar x_{1 \pm},\bar x_{2 \pm}}$
\be \begin{split} \label{omro4}
\sqrt \om \ = & \ \f{2 \rho + 4 + \sqrt{\rho^2 + 4 \rho + 16}}{6} \\
\qquad E_{\delta'}(\varphi_{\rho,-}^{\bar x_{1 \pm},\bar x_{2 \pm}}) \ = & \ -
\f 5 {216} \rho^3 - \f 1 {54}
( \rho^2 + 4 \rho + 16)^{\f 3 2} - \f 5 {36} \rho^2 - \f 4 9  \rho - \f
{32} {27}.
\end{split} \ee
\end{itemize}
So, the first inequality in \eqref{ineq1} is immediately proven. For the
second inequality, one  first extends by continuity the functions 
$ E_{\delta'}(\varphi_{\rho,+}^{0,0})$ and
$E(\varphi_{\rho,+}^{x_{1 \pm},x_{2 \pm}})$ up to 
the value $\rho = 0$ and obtains
$$
\f {d^j} {d \rho^j} 
E_{\delta'}(\varphi_{0+,+}^{0,0}) \ = \ \f {d^j} {d \rho^j}
E_{\delta'}(\varphi_{0+,+}^{x_{1 \pm},x_{2 \pm}})  \ = \ 
0, \qquad j = 0, 1, 2.
$$
Furthermore, 
$$\f {d^3} {d \rho^3} 
E_{\delta'}(\varphi_{0+,+}^{0,0}) \ = \ \f {d^3} {d \rho^3}
E_{\delta'}(\varphi_{0+,+}^{x_{1 \pm},x_{2 \pm}})  \ = \ - \f 1 {16}.$$
The second inequality in \eqref{ineq1} is then proven by
$$ 0 \ = \  \f {d^4} {d \rho^4} 
E_{\delta'}(\varphi_{\rho,+}^{0,0}) \ < \ \f {d^4} {d \rho^4}
E_{\delta'}(\varphi_{\rho,+}^{x_{1 \pm},x_{2 \pm}})  \ = \ 12 (\rho^2
+ 4 \rho + 16)^{-\f 5 2}.$$ 

It remains to prove the first inequality in \eqref{ineq2}. First,
observe that
$$
E_{\delta'}(\varphi_{8(\sqrt 2 -1),-}^{\bar x,- \bar x}) \ = \ 
E_{\delta'}(\varphi_{8(\sqrt 2 -1),-}^{\bar x_{1 \pm},\bar x_{2 \pm}})  \ = \ -
\f {16} 3 (2 \sqrt 2 - 1) 
$$
and
$$
\f d {d \rho} 
E_{\delta'}(\varphi_{8(\sqrt 2 -1),-}^{\bar x,- \bar x}) \ = \ \f d {d \rho} 
E_{\delta'}(\varphi_{8(\sqrt 2 -1),-}^{\bar x_{1 \pm},\bar x_{2 \pm}})  \ = \ - 4.
$$
Now, we consider the second derivative. The statement $\f  {d^2} {d
  \rho^2} E_{\delta'}(\varphi_{\rho,-}^{\bar x,- \bar x}) \ > \ \f {d^2} {d \rho^2} 
E_{\delta'}(\varphi_{\rho,-}^{\bar x_{1 \pm}, \bar x_{2 \pm}})$ proves equivalent to
$(\rho^2 + 4 \rho^{\f 1 2} + 16)^{\f 1 2} (\rho^2 + 4 \rho^{\f 1 2} +
10) \ > \ - \f {11}{16} \rho + \f 2 9$, which is satisfied for any
$\rho \geq 8 (\sqrt 2 - 1)$. The proof is complete.

\end{proof}

\subsection{Linear limit for the stationary states}
\begin{theorem}
For $\lambda \to 0+$, the following limits hold:
\begin{eqnarray} \label{erste} 
\varphi_{\rho,+}^{0,0} & \longrightarrow  & 0, \qquad {\mbox{in the weak
    topology of}} \, ~ Q_{\delta'}  \\ \label{zweite}
\varphi_{\rho,+}^{x_{1 \pm}, x_{2 \pm}}  & \longrightarrow  &  0, \qquad {\mbox{in the weak
    topology of}} \, ~ Q_{\delta'}  \\ \label{dritte}
\varphi_{\rho,-}^{- \bar x, \bar x}  & \longrightarrow  & \sqrt{\f {2
    \rho} \gamma} e^{- \f 2 \gamma | \cdot | } , \qquad {\mbox{in the strong
    topology of}}  \, ~ Q_{\delta'} 
\end{eqnarray}
\end{theorem}
\begin{proof}
Limit \eqref{erste} is immediately proven once observed that,
afer rescaling \eqref{scaling}, from \eqref{omro1} one has $\sqrt \om = \f {\lambda
  \rho}
{4 \sqrt \gamma},$ so that $\sqrt \om$ vanishes in the limit.

\noindent
To prove \eqref{zweite} notice that
the first formula in \eqref{omro2}, after rescaling \eqref{scaling}, gives
$$
\sqrt \om \ = \ \f{2 \gamma \lambda \rho + 4 - \sqrt{\gamma^2
    \lambda^2 \rho^2 + 4 \gamma \lambda \rho + 
    16}}{6 \gamma}, 
$$
so that $ \sqrt \om \approx \f {\lambda
  \rho}
{4 \sqrt \gamma},$ as in the previous case.

Limit \eqref{dritte} is mapped into limit \eqref{linlimdelta} by the
replacement $\gamma \to \f 4 \alpha$, so the reader is referred to the
proof of Theorem \ref{22}.

\end{proof}
\begin{remark} 
Notice that the nonlinear limit for the stationary states
$\varphi_{\rho,-}^{\bar x_{1 \pm},\bar x_{2 \pm}}$ is nonsensical: as $\lambda$ vanishes,
the mass $\rho$ of the state is overcome by the threshold $\f 8
{\lambda \gamma}
(\sqrt 2 -1)$ under which there is no such a state. In other
words, in order for  $\varphi_{\rho,-}^{\bar x_{1 \pm},\bar x_{2 \pm}}$ to exist, the
nonlinearity must be not only
present, but also sufficiently strong.
\end{remark}

\begin{remark}
A closer look to the limit \eqref{zweite} shows that, in such case,
both $x_{1+}, x_{2+}$ diverge positively (analogously   $x_{1-},
x_{2-}$ diverge negatively) as $\lambda \to 0+$ (see \eqref{313},
\eqref{solt+}). 
As a consequence, in the
nonlinear limit the corresponding stationary state
$\varphi_{\rho,+}^{x_{1 \pm},x_{2 \pm}}$ (or, equivalently,
$\psi_{\om,+}^{x_{1 \pm},x_{2 \pm}2}$) takes the shape of a soliton that runs 
towards the infinity. Conversely, this means that, starting from the
linear dynamics generated by $H_2$ and turning on the nonlinearity, 
two stationary states arises from very far away and approach the
origin as $\lambda$ grows.  
\end{remark}

\begin{remark}
Limit \eqref{dritte} shows that the nonlinear symmetric
changing sign stationary state  $\varphi_{\rho,-}^{\bar x,-\bar x}$
bifurcates from the linear ground state of the Hamiltonian $H_2$.
\end{remark}

\subsection{No-defect limits}
As stressed in Section 1, we perform two kinds of no-defect limits.

\begin{theorem}
For $\gamma \to 0+$, the following limits hold in the space
$Q_{\delta'}$ defined in 
\eqref{qdelta'}: 
\begin{eqnarray}     \label{primeiro}
\varphi_{\rho,+}^{x_{1 \pm}, x_{2 \pm}}  & \longrightarrow  &  \phi
\left( \sqrt{\f \lambda 2} \f \rho 2, \f {\lambda \rho } 4, \f 2
     {\lambda \rho} \pm \log (3 + 2 \sqrt 2)
\right)
, ~ {\mbox{in the strong
    topology of}} \, ~ Q_{\delta'}  \\ \label{terceiro}
\varphi_{\rho,-}^{- \bar x, \bar x}  & \longrightarrow  & 0
, \qquad
       {\mbox{in the sense of ditributions
    }}  
\end{eqnarray}
\end{theorem}
\begin{proof}
Limit \eqref{primeiro} is obtained by observing that, by \eqref{omro2}
with the rescaling \eqref{scaling}, $\omega \approx \f {\lambda \rho} 4
$ as $\lambda \to 0+$. As a consequence, by \eqref{solt+} and
\eqref{stazzionario2+}, one gets $x_{1 \pm} \to \pm \f 2
     {\lambda \rho} \log (3 + 2 \sqrt 2)$. \eqref{primeiro} then follows from
     Lemma \ref{prelemma}.

To prove \eqref{terceiro} we just notice that, from the first identity
in \eqref{omro3}, $\om \to \infty$ and, through \eqref{scaling},
$\gamma \om \to 2$. Then \eqref{xbar} shows that $\bar x \to \infty$.
\end{proof}

\begin{remark}
It is immediate to observe that the stationary state $\varphi_{\rho,+}^{0,0}$
is untouched as $\gamma$ varies, while $\varphi_{\rho,-}^{\bar x_{1
    \pm},\bar x_{2 \pm}}$ ceases to exist as $\gamma \leq \f 8
{\lambda \rho} ( \sqrt 2 - 1)$.
\end{remark}

\begin{remark}
Limit \eqref{primeiro} shows that the stationary states
$\varphi_{\rho,+}^{x_{1 \pm}, x_{2 \pm}} $ bifurcate from the nonlinear soliton
 $\phi
\left( \sqrt{\f \lambda 2} \f \rho 2, \f {\lambda \rho } 4, \f 2
     {\lambda \rho} \log (3 + 2 \sqrt 2) \right)$ at $\rho > 0$.
\end{remark}

\begin{theorem}
For $\gamma \to +\infty$, the following limits hold in the space
$Q_{\delta'}$ defined in 
\eqref{qdelta'}: 
\begin{eqnarray}     \label{primo}
\varphi_{\rho,+}^{x_{1+}, x_{2+}}  & \longrightarrow  &
\chi_{(-\infty,0]} \, \phi
\left( \sqrt{\f \lambda 2} \f \rho 3, \f {\lambda \rho } 6, 0
\right)
, \qquad {\mbox{in the weak
    topology of}} \, Q_{\delta'}  \\ \label{secondo}
\varphi_{\rho,-}^{- \bar x, \bar x}  & \longrightarrow  &  \phi
\left( \sqrt{\f \lambda 2} \f \rho 2, \f {\lambda \rho } 4, 0
\right)
, \qquad
       {\mbox{in the strong
    topology of}} \, Q_{\delta'}  \\ \label{terzo}
\varphi_{\rho,-}^{\bar x_{1+}, \bar x_{2+}}  & \longrightarrow  &
\chi_{[0,+\infty)} \, \phi
\left( \sqrt{\f \lambda 2} \rho, \f {\lambda \rho } 2, 0
\right)
, \qquad {\mbox{in the strong
    topology of}} \, Q_{\delta'}     
\end{eqnarray}
where $\chi_I$ denotes the characteristic function of the set $I$.
\end{theorem}
\begin{proof}
To prove
the limit \eqref{primo}, notice first that, from \eqref{omro2} and the
scaling law \eqref{scaling} one has $\sqrt \om \to \f {\lambda \rho}
6$. Then, \eqref{solt+} shows that $t_1 \to 0$ as $t_2 \to 1$, so it
must be $x_{1+} \to 0$ and  $x_{2+} \to \infty$. Lemma
\eqref{prelemma} shows that in the left halfline the convergence is
strong in the space $H^1 (\erre^{-})$, while in the positive halfline
there is a full soliton that run aways, weakening the convergence to
zero.

Limit \eqref{secondo} is an immediate consequence of \eqref{omro3} and
\eqref{xbar} through Lemma \eqref{prelemma}.

Finally, limit \eqref{terzo} is proven through \eqref{omro4}, that
shows $\sqrt \om \to \f \rho 2$, 
\eqref{solt-asym}, that gives $t_{1,+} \to 1$ and $t_{2,+} \to 0$, and
Lemma \eqref{prelemma}. 
 \end{proof}

\end{document}